\documentclass[12pt,a4paper]{amsart}

\usepackage{amsfonts, amsmath, amssymb, amscd, latexsym, graphicx}

\def\zz{\mathbb{Z}}

\newcommand{\ft}{\ensuremath{F_\tau}}
\newcommand{\ttt}{\ensuremath{T_\tau}}
\newcommand{\vt}{\ensuremath{V_\tau}}

\newcommand{\txz}{\ensuremath{T_{xz}}}
\newcommand{\vxz}{\ensuremath{V_{xz}}}
\def\fb{F_\beta}
\def\vb{V_\beta}

\def\eps{\varepsilon}

\newcommand{\quot}[2]{{\left.\raisebox{.2em}{$#1$}\middle/\raisebox{-.2em}{$#2$}\right.}}

\def\qed{\hfill\rlap{$\sqcup$}$\sqcap$\par}

\newtheorem{thm}{Theorem}[section]
\newtheorem{prop}[thm]{Proposition}
\newtheorem{lem}[thm]{Lemma}
\newtheorem{cor}[thm]{Corollary}

\newtheorem{defn}[thm]{Definition}

\newtheorem{question*}{Question}

\title{Irrational-slope versions of Thompson's groups $T$ and $V$}
\author{Jos\'e Burillo }
\address{Departament de Matem\`atiques,
Universitat Polit\`ecnica de Catalunya,
Jordi Girona 1-3,
08034 Barcelona, Spain}
\email{pep.burillo@upc.edu}

\author{Brita Nucinkis}
\address{Department of Mathematics, Royal Holloway, University of London,
Egham, TW20 0EX, UK}
\email{Brita.Nucinkis@rhul.ac.uk}

\author{Lawrence Reeves}
\address{School of Mathematics and Statistics,
University of Melbourne,
VIC 3010, Australia}
\email{lreeves@unimelb.edu.au}

\date{}

\begin{document}

\maketitle

\begin{abstract} In this paper we consider the $T$- and $V$- versions, $\ttt$ and $\vt$, of the irrational slope Thompson group $\ft$ considered in \cite{ftau}. We give infinite presentations for these groups and show how they can be represented by tree-pair diagrams similar to those for $T$ and $V.$ We also show that $\ttt$ and $\vt$ have index-$2$ normal subgroups, unlike their original Thompson counterparts $T$ and $V.$ These index-$2$ subgroups are shown to be simple.

\end{abstract}

\section*{Introduction} This paper is a continuation of the authors' previous paper \cite{ftau}, where they
studied in detail the irrational-slope Thompson's group \ft\ introduced by Cleary in \cite{clearyirr}, a group of piecewise linear maps of the interval with irrational slopes that are powers of the small golden ratio
$\tau=(\sqrt{5}-1)/2$
and breakpoints in the ring $\zz[\tau]$. In the standard fashion in the family of Thompson's groups, that study is extended here to the $T$ and $V$ versions, analogously defined via maps on the circle or left-continuous maps of the interval. The traditional Thompson's groups $T$ and $V$ are simple, but the irrational-slope versions  have subgroups of index two. This is phenomenon stems from the special relation in the group, which  introduces 2-torsion in the abelianisation of \ft. Finally, we mention the possibility of taking a $V_3$-version with irrational slopes to construct Thompson-like groups with normal subgroups of index 4.

The paper is organised as follows: We begin by defining the groups and indicate how to work with their elements. Then we go into the specifics for each group, beginning with \ttt\, where we give a presentation and the proof of the simplicity of the index-two subgroup. Then follows a section for \vt\, which also contains a presentation and the proof of the simplicity of the index-two subgroup, although the proof of simplicity of that subgroup is substantially different than that for \ttt. We also explore the possibility of the index-four subgroups for the three-caret version.

Also note that these groups are of type $F_\infty.$ The result for $\ft$ was proven in \cite{clearyirr}, see \cite{clearyirrold} for detail. The extension of the result to $\ttt$ and $\vt$ follows by directly applying the methods of Stein \cite{stein}. These methods are by now standard, see, for example, \cite{FMWZ, KoMN4, MMN}, and we will not present them here.

 This paper is the natural continuation of \cite{ftau}, and many results and definitions are taken directly from that paper.

\section{The groups \ttt\ and \vt\ }

Let $\tau$ be the small golden ratio, namely $\tau=(\sqrt5-1)/2=0.618...$, which satisfies that $1=\tau+\tau^2$. In \cite{clearyirr}, Cleary introduced the group \ft, the irrational Thompson group, with breaks in $\zz[\tau]$ and slopes powers of $\tau$. From the equality $1=\tau+\tau^2$, we subdivide an interval of length $\tau^n$ into one of length $\tau^{n+1}$ and one of length $\tau^{n+2}$. The group \ft\ is the group of piecewise linear maps on the interval with these breaks and slopes, see \cite{ftau} for details.

As it is standard in the Thompson family, we construct the irrational slope versions corresponding to the groups $T$ and $V$. For the group \ttt, one can consider maps on the circle instead of the interval. The circle will be obtained by identifying the two endpoints of the interval $[0,1]$, so we can consider maps of the interval such that the images of 0 and 1 coincide. So the group \ttt\ is the group of piecewise-linear, orientation-preserving homeomorphisms of the circle such that the breakpoints are in $\zz[\tau]$ and the slopes of the linear parts are powers of $\tau$. The $V$-version \vt\ consists of the left-continuous, piecewise-linear maps of $(0,1]$, also with breaks in $\zz[\tau]$ and slopes powers of $\tau$. See \cite{cfp} for details of this construction for $V$, which is analogous to this one. See Figure \ref{telt} for elements in \ttt\ and \vt.

\begin{figure}
  \centerline{\includegraphics[width=120mm]{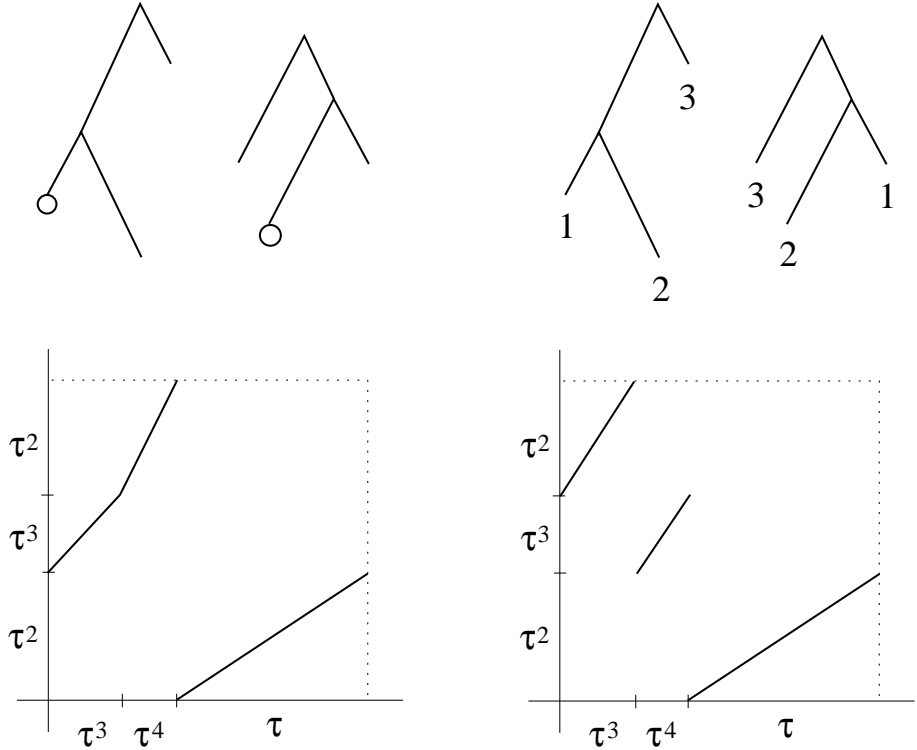}}
  \caption{Elements in \ttt\ and \vt, based on the same subdivisions. In the element in \ttt\ the circles mark that the first leaf is mapped to the second leaf (and the rest in cyclic order), while in the element in \vt\ the numbers indicate the permutation and the way leaves are mapped.}\label{telt}
\end{figure}

Thompson's group $V$ is also commonly seen as homeomorphisms of the Cantor set, but this interpretation is not possible here. The fact that we have two types of carets and we have subdivisions which can be obtained in more than one way makes this interpretation not very clear. The boundary is not well defined the traditional way, because of the different tree configurations that give the same subdivisions. Hence we will stick with left-continuous maps, but we will favour all throughout the paper the expression of the elements in tree-permutation form, as is explained in what follows.

It was shown in \cite {clearyirr} that all elements of $\ft$ can be obtained by pairs of iterated subdivisions of the unit interval as above.
By subdividing an interval of length $\tau^n$ into two intervals of lengths $\tau^{n+1}$ and $\tau^{n+1},$ we have two options, either put the short interval, or the long interval first. When representing iterated subdivisions of the interval by a tree, we can use two different types of carets: the $x$-carets have a long leg on the left, representing the short first interval, and the $y$-carets have their left leg short, representing the long first interval. See \cite{ftau} for details for the representation of elements of $\ft$ by tree-pair diagrams.

We now represent the elements of $\ttt$ and $\vt$ by the same tree-pair diagrams but with an added permutation of the leaves of the trees, see \cite{cfp} for the analogue description for $T$ and $V$.
For the group \ttt\ these permutations preserve the cyclic order of the leaves, hence giving maps of the circle. For the group \vt\ any permutation is allowed in the leaves. This is represented by labels on the leaves indicating the permutation, although the cyclic permutations in \ttt\ are determined by the image of the first leaf, so we may abbreviate the permutation with just a circle in the image of the first leaf.

The diagrams are usually denoted by $(T_1,\pi,T_2)$, where $T_1$ and $T_2$ are trees and $\pi$ is the permutation of the leaves.

Results from \cite{ftau} can be used the same way to compose elements. The target tree of the first element has to be matched with the source tree of the second element, as is standard in the Thompson family. To do this, one may need to change the type of some carets, but this is independent of the two permutations of the elements, so it can be done exactly as in \ft, using Proposition 2.1 and Lemma 2.2 in \cite{ftau}. The only difference is that if a caret is added, the counterpart caret added in the other tree has to go into the corresponding leaf according to the permutation.

A reader familiar with \ft\ and with Thompson's groups $T$ and $V$ will have no problem with all these definitions, as they are very similar to them.

Finally, we can use the methods developed in \cite{ftau} to construct a particularly appropriate diagram for elements of the groups \ttt\ and \vt. As \ttt\ is a subgroup of \vt, the lemma is stated for elements of \vt.

\begin{lem}\label{el-vt-lem} Let $v \in \vt.$ Then there exists a triple $(T_1, \pi, T_2)$ representing $v$, where
\begin{enumerate}
\item $T_1$ and $T_2$ are trees with the same number of leaves, $k$ say.
\item $\pi \in \mathcal S_k.$
\item $T_2$ is a tree consisting entirely of $x$-carets.
\item  The only $y$-carets in $T_1$ have no left children.
\end{enumerate}
\end{lem}

\begin{proof} The proof is entirely analogous to that of \cite[Lemma 6.1]{ftau}, with the exception of including a pair of spines in the middle, in which the leaves are permuted. When adding carets to $T_1$ and $T_2$ simultaneously, one needs to take account of the permutation of leaves, which does not change the outcome of the procedure.
\end{proof}

\begin{cor}\label{SNF-cor}
Every element in $v \in \vt$ has an expression
$$v = x_0^{a_0}y_0^{\eps_0}....x_n^{a_n}y_n^{\eps_n}\pi x_m^{-b_m}....x_0^{-b_0},$$
where $a_i,b_j,n,m \in \zz_{\geq 0}$ ($i=0,...,n$; $j=0,...m$), $\pi \in \mathcal S_k$ for some $k,$ and $\eps_i \in \{0,1\}.$
\end{cor}

This expression is called, in short, the $p\pi q^{-1}$ form of an element. The only particularity satisfied by elements of \ttt\ is that the permutation obtained will be a cyclic permutation $c$ of the leaves, and in this case, we call it the $pcq^{-1}$ form, to keep with the tradition. The $p\pi q^{-1}$ form can be used to directly solve the word problem for the element by using the following lemma.

\begin{lem}\label{vt-identity-lem} Suppose the identity element is given by an expression $p\pi q^{-1},$ where $p$ and $q$ are positive words in the $x_i,y_i$, $i \geq 0,$ and $\pi \in \mathcal S_n$ for some $n$. Then $\pi = id_{\mathcal S_n}.$
\end{lem}

\begin{proof} We represent $id$ by a reduced tree-pair diagram $(T_1,T_2),$ where the leaves of $T_1$ are labelled in order, and the labelling of the leaves of $T_2$ is given by $\pi.$ We show that the leaves of $T_2$ are also labelled in order; hence $\pi =1.$

Since $id(0)= 0, $ this implies that leaf $1$ of $T_1$ is mapped to the left-most leaf of $T_2$. Furthermore, the slope of $id$ is $1$, and hence the depths of leaves $1$ in $T_1$ and $T_2$ are the same, and hence the next leaves in each
tree represent the same break-point $x \in \zz[\tau] \cap [0,1].$ Now $id(x) =x,$ and hence the second leaf in $T_2$ is also labelled by $2$, and, since the slope is $1$, also of the same level as leaf $2$ in $T_1.$ We continue this argument to see that $\pi =1.$
\end{proof}

From this lemma, it is straightforward to solve the word problem for the groups \ttt\ and \vt: given an element, find its diagram and its $p\pi q^{-1}$ form. If $\pi$ is not the identity, the element is not the identity. And if the permutation is the identity, then the element is actually in \ft, and that word problem was solved in \cite{ftau}.

Finally, just observe that the method we will follow to obtain presentations for \ttt\ and \vt\ is to see that the $p\pi q^{-1}$ form can be obtained algebraically with the set of generators stated. Hence, we proceed with the details for both groups.

\section{A presentation for \ttt}

\subsection{Generators for \ttt}

As happens with the standard dyadic groups, the only thing one needs to do to obtain generating sets for $T$ is to add cycles. Recall that in  \ft\ we gave preference to $x$-type carets, where the left leg was long and the right one was short. Generators were constructed with spines made out of $x$-carets.  The $y_n$ generators had one $y$-type caret in the only nonspine position. So it only makes sense here to consider cyclic permutations of a spine as generators, keeping the methods of \cite{cfp}. Hence, the generators $c_n$, for $n\ge 1$, will be defined as an element whose two trees are spines with $n+1$ carets, and where the first leaf is mapped to the last leaf. This is an element of order $n+2$ (or one of its divisors), since it is a cycle on its $n+2$ leaves. See Figure \ref{cgens} for the diagrams of the $c$-generators.

\begin{figure}
  \centerline{\includegraphics[width=60mm]{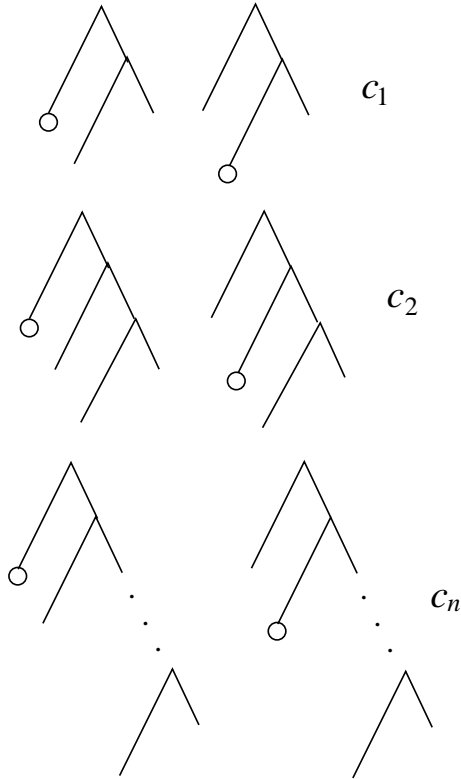}}
  \caption{The $c$ generators in \ttt.}\label{cgens}
\end{figure}

From here we can already find a set of generators.

\begin{prop}
The set $\{x_n,y_n,c_n\mid n\ge 1\}$ is a generating set for the group \ttt.
\end{prop}

{\it Proof.} To see this one only needs to decompose a tree diagram in three: a diagram for a positive element from \ft, a cycle (i.e., a power of a $c_n$) and a negative element of \ft. This can be done easily by introducing two spines in the middle of the diagram: the positive piece is given by the first tree and a spine (with the same number of carets), then a cycle with this same spine, which concentrates the permutation part of the element, and finally the spine with the target tree. Since the \ft\ elements can be generated by the $x_n$ and $y_n$, and the central cycle is a power of a $c_n$, we have a generating set.\qed

We need to evaluate the interactions of the generators with each other to find the relations. Clearly all relations from \ft\ are satisfied here the same way, so we need to see how the $x_n$ and $y_n$ generators interact with the $c_n$. Observe, to start with, that the generators $x_n$ and $c_n$ by themselves generate a copy of $T$ inside \ttt, because all elements have only $x$-type carets involved and the combinatorics are exactly the same as those in $T$. Hence, the relations on $T$ between $x_n$ and $c_n$ are satisfied the same way. So we already have the following relations:

\begin{enumerate}
\item[(1)] The relators from \ft:
\begin{enumerate}
\item[(1.1)] $x_jx_i=x_ix_{j+1}$, if $i<j$.
\item[(1.2)] $x_jy_i=y_ix_{j+1}$, if $i<j$.
\item[(1.3)] $y_jx_i=x_iy_{j+1}$, if $i<j$.
\item[(1.4)] $y_jy_i=y_iy_{j+1}$, if $i<j$.
\item[(1.5)] $y_n^2=x_nx_{n+1}$.
\end{enumerate}
\item[(2)] The relations involving the generators $c_n$ and $x_n$,  as in $T$:
\begin{enumerate}
\item[(2.1)] $x_kc_{n+1}=c_{n}x_{k+1}$, if $k<n$.
\item[(2.2)] $c_nx_0=c_{n+1}^2$.
\item[(2.3)] $c_n=x_nc_{n+1}$.
\item[(2.4)] $c_n^{n+2}=1$.
\end{enumerate}
\end{enumerate}

We only need to find analogs for the relations (2.1), (2.2) and (2.3) between the $y_n$ and $c_n$ by checking how these generators interact. It is straightforward to check that the relations $y_kc_{n+1}=c_ny_{k+1}$ are satisfied for all $k<n$, because the $y$-caret does not affect the cycle, they work in the exact same way as the $x$-carets. To find an analog of (2.2) for the $y_n$, we observe that the product $c_ny_0$ acquires a $y$-caret in the last leaf, continuing the spine, and this must be eliminated to rewrite it in the generators. After adding a caret and performing a basic move, we obtain the product $y_{n+1}^{-1}c_{n+1}^2$. Finally, since normal forms for the elements of \ft\ do not have any appearances of $y_n^{-1}$, a relator analog to (2.3) will not be needed.

So, to those families of generators, we can add the following two:

\begin{enumerate}
\item[(3.1)] $y_kc_{n+1}=c_{n}y_{k+1}$, if $k<n$.
\item[(3.2)] $c_ny_0=y_{n+1}^{-1}c_{n+1}^2$.
\end{enumerate}

It is easy to check by direct inspection that these relations are all true in \ttt. The goal of the following section will be to prove that this gives a presentation for the group.

\subsection{Relators for \ttt}

The goal of this section is to exhibit a presentation of \ttt. The theorem we will prove is the following.

\begin{thm}\label{prez} A presentation for the group \ttt\ is given by the generators
$x_n$, $y_n$ and $c_n$, for $n\ge 1$ with the following relators:
\begin{enumerate}
\item[(1)] The relators from \ft:
\begin{enumerate}
\item[(1.1)] $x_jx_i=x_ix_{j+1}$, if $i<j$.
\item[(1.2)] $x_jy_i=y_ix_{j+1}$, if $i<j$.
\item[(1.3)] $y_jx_i=x_iy_{j+1}$, if $i<j$.
\item[(1.4)] $y_jy_i=y_iy_{j+1}$, if $i<j$.
\item[(1.5)] $y_n^2=x_nx_{n+1}$.
\end{enumerate}
\item[(2)] The relations involving the generators $c_n$ and $x_n$, same as in $T$:
\begin{enumerate}
\item[(2.1)] $x_kc_{n+1}=c_{n}x_{k+1}$, if $k<n$.
\item[(2.2)] $c_nx_0=c_{n+1}^2$.
\item[(2.3)] $c_n=x_nc_{n+1}$.
\item[(2.4)] The finite order for the $c_n$, namely, $c_n^{n+2}=1$.
\end{enumerate}
\item[(3)] The relations involving the generators $c_n$ and $y_n$:
\begin{enumerate}
\item[(3.1)] $y_kc_{n+1}=c_{n}y_{k+1}$, if $k<n$.
\item[(3.2)] $c_ny_0=y_{n+1}^{-1}c_{n+1}^2$.
\end{enumerate}
\end{enumerate}
\end{thm}

The main tool to prove this theorem is the method to write any element in $pcq^{-1}$ form. This is the algebraic analog of the splitting of a diagram in three pieces used to show that these generators suffice, explained at the beginning of the previous section.

\begin{prop}\label{2.3} Every element of $T_\tau$ can be written in $pcq^{-1}$ form, where $p$ and $q$ are positive elements of \ft\ and $c$ is short for $c_n^m$ for some $m,n$.
\end{prop}

{\it Proof.} We need the following well-known lemma, which is known to hold in $T$:

\begin{lem}\label{2.4} Let n be a positive integer, The generators $x_n$ and $c_n$ satisfy the following equalities:
\begin{enumerate}
\item[(i)] $c_n^m=x_{n+1-m}c_{n+1}^m$, for $m\le n+1$
\item[(ii)] $c_n^m=c_{n+1}^{m+1}x_{m-1}^{-1}$
\end{enumerate}
\end{lem}

This is contained in Lemma 5.6 in \cite{cfp}. \qed

We continue with the proof of our proposition. Given a word in the $x_n,y_n,c_n$ generators, we want to transform it into $pcq^{-1}$ form. First of all, in this word, between two consecutive $c_n$ generators there is a word only in $x_n$ and $y_n$, representing an element in $\ft.$  According to \cite{ftau} we can put this element of \ft\ in seminormal form
$$
x_0^{a_0}y_0^{\epsilon_0}x_1^{a_1}y_1^{\epsilon_1}\dots x_n^{a_n}y_n^{\epsilon_n} x_m^{-b_m}x_{m-1}^{-b_{m-1}}\dots x_1^{-b_1}x_0^{-b_0}
$$
where $a_i,b_i\ge 0$ and $\epsilon_i\in\{0,1\}$.  The proof of Proposition \ref{2.3} will follow from a repeated application of the following process: transform
$$
c_n^m
x_0^{a_0}y_0^{\epsilon_0}x_1^{a_1}y_1^{\epsilon_1}\dots x_n^{a_n}y_n^{\epsilon_n} x_m^{-b_m}x_{m-1}^{-b_{m-1}}\dots x_1^{-b_1}x_0^{-b_0}
c_k^l
$$
into $w(x_n,y_n)c_N^M$, where $w(x_n,y_n)$ is a word in the $x_n$ and $y_n$ only, and $N,M$ are appropriately chosen positive integers.  Applying this process repeatedly from left to right will yield the result.

To move the $c_n^m$ past the $x_0$ use the standard $T$ relations. Observe that the index of a given $c_n^m$ can be raised as much as we need, at the price of adding some instances of $x_n$ to its left (which is fine) by repeated uses of property (i) of Lemma \ref{2.4}. To move an element $c_n^m$ past $y_0$, we apply:
$$
c_n^m y_0=c_n^{m-1} c_n y_0=c_n^{m-1} y_{n+1}^{-1} c_{n+1}^2
$$
after using relation (3.2). To move the $y_{n+1}^{-1}$ to the left past the $c_n^{m-1}$, we will take inverses. If we take inverses directly as $y_{n+1}c_n^{1-m}$, to apply (3.1) we need to have raised the index of $c_n$ beforehand, because on the left hand side of (3.1) the term $y_kc_{n+1}$ requires that the index of $c$ is at least two higher than that of $y$. Since there are several instances of (3.1) to be taken, and each one increases the index of $y$ by 1, we will need to raise the index several times, using (i) from the lemma. Hence:
$$
\begin{array}{ll}
c_n^{m-1} y_{n+1}^{-1}&=x_{n-m+2} c_{n+1}^{m-1} y_{n+1}^{-1}=x_{n-m+2}x_{n-m+3}c_{n+2}^{m-1}y_{n+1}^{-1}\\
&=x_{n-m+2}x_{n-m+3}\ldots x_Nc_N^{m-1}y_{n+1}^{-1}
\end{array}
$$
where $N$ is an index large enough, much larger than $n+1$, to be able to apply the following:
$$
c_N^{m-1}y_{n+1}^{-1}=[y_{n+1} c_N^{N-m+3}]^{-1}
$$
and now use repeated applications of (3.1) to obtain
$$
[y_{n+1} c_N^{N-m+3}]^{-1}=[c_{N-1}^{N-m+3} y_{n+N-m+2}]^{-1}.
$$
This yields the equality
$$
c_n^m y_0=x_{n-m+2}\ldots x_N y_{n+N-m+2}^{-1} c_{N-1}^{N-m+3}
$$
which is what we desired to get the $c_n^m$ past the $y_0$.

To get the $c_n^m$ past $x_k$, use $T$ again. And to go past $y_k$, note that now we can use relation (3.1) directly. Raise the index of the $c_n^m$ if necessary using (i) in the lemma and when the index is high enough, reverse them using (3.1). If the index of $y_k$ becomes zero, use the algorithm above.

Ultimately, this produces an element with a word on $x_n$ and $y_n$, followed by a product $c_n^m c_k^l$. To transform this into a power of a single $c$-generator, use Lemma \ref{2.4} to raise the smaller index of $n$ and $k$, using (i) if $n$ needs to be raised, and (ii) if $k$ is smaller. The only penalty of doing this is introducing $x$-generators to the left or right of the pair of $c$-generators, which works well for our purposes. This finishes the process of getting $c_n^m$ past an element of \ft\ and merging it into the following $c_k^l$. Continue this process until only one power of a $c_n$ is left.

This produces an element of the type $pq^{-1}crs^{-1}$, where $p,q,r,s$ are positive words and we can assume (by the properties of \ft) that $q$ and $s$ contain only $x$-generators. To finalize the process, move the word $r$ to the left using the method above, leaving only negative elements to the right of $c$ in a word $wcs$, where now $w$ clearly can have negative elements. Put $w$ in \ft-normal form and finally use $T$ to move the negative $x$-generators to the right of $c$. This finishes the proof of Proposition \ref{2.3}. \qed

This process constructs an easy expression for an element which can be used to solve the word problem and to show that these relators suffice to have a presentation. To prove that all relators in \ref{prez} are enough to give a presentation, observe that given a word in the $x_n,y_n,c_n$ which represents the identity in \ttt, rewrite it in $pcq^{-1}$ form. Since the word represents the identity, apply Lemma \ref{vt-identity-lem} to see that $c$ is actually the identity. And then this means the original element was in \ft\ and is a product of conjugates of the relators for this group, which are in the presentation as well.

This process finishes the proof of Theorem \ref{prez}.

\section{The group \txz\ is simple}

The appearance of the relator (1.5), present already in \ft, was the reason for the 2-torsion observed in the abelianization of \ft. Observe that, as it happens in \ft, the parity of the $y$-generators is preserved in all relations of \ttt, see Theorem \ref{prez}. This shows that there is a homomorphism from \ttt\ onto $\quot{\zz}{2\zz}$, which sends $x_n$ and $c_n$ to 0 and $y_n$ to 1.
\begin{defn} We call $\txz$ the kernel of the map
$ \varphi: \ttt \longrightarrow \quot{\zz}{2\zz}.$
\end{defn}
The group \txz\ is generated by the $x_n$, $c_n$ and by the new generators $z_n=y_{2n}y_{2n+2}$. Hence, the group \ttt\ has no chance of being simple like $T$, but this kernel is.

\begin{thm} The group \txz\ is simple.
\end{thm}

{\it Proof.} Let $N$ be a normal subgroup of \txz\, and let $a\neq 1$ be an element in $N$. We want to show that $N=\txz$. Let
$$
\theta:\txz\longrightarrow\quot\txz{N}
$$
be the canonical quotient map. Since $a$ is nontrivial, we know it can be written in the form $pcq^{-1}$, and we only need that $p,q\in\ft$, the parity of $y$ is irrelevant for this argument. Since $\theta(a)=1$, we have that $\theta(p^{-1}q)=\theta(c)$, and hence, since $c$ is of finite order, we have that $\theta((p^{-1}q)^n)=1$ for some $n$. We now have two possibilities:
\begin{enumerate}
\item $p^{-1}q\neq1$. If this is the case, and since \ft\ is torsion-free, we have that $(p^{-1}q)^n$ is an element which is in \ft, it is not the identity but it lies in the kernel of $\theta$.
\item $p^{-1}q=1$. Then $\theta(c)=1$. Suppose $c=c_n^m$. Use relation (2.3) to see that $c_n^m=c_n^{m-1}x_nc_{n+1}$ and applying relator (2.1) $m-1$ times, we get
$c_n^m=x_{n-(m-1)}c_{n+1}^m$.
Hence, since $\theta(c_n^m)=1$, we have that $\theta(x_{n-(m-1)})=\theta(c_{n+1}^{-m})$. But this means that $\theta(x_{n-(m-1)}^{n+3})=1$, and again we have found a nontrivial element in \ft\ and in the kernel of $\theta$.
\end{enumerate}
Restrict $\theta$ to \ft\ and recall that a quotient of \ft\ is either isomorphic to \ft, or else is abelian. From the two cases above we deduce that $\theta|_{\ft}$ is not an isomorphism, so we conclude that $\theta(\ft)$ is an abelian group.

From here we have that the images by $\theta$ of all $x_n$ and $y_n$ commute. As is done in \cite{cfp} this means that:
\begin{enumerate}
\item From relation (2.1) with $n=2$, $k=1$, namely, $x_1c_3=c_2x_2$, that is, $x_1(x_0^{-2}c_1x_1^2)=(x_0^{-1}c_1x_1)(x_0^{-1}x_1x_0)$ we deduce that $\theta(x_1)=\theta(x_0)$.
\item From relation (2.3) with $n=1$, namely, $c_1=x_1(x_0^{-1}c_1x_1)$, we deduce that $\theta(x_0)=\theta(x_1)=1$.
\item From relation (2.2) with $n=1$, namely, $c_1x_0=(x_0^{-1}c_1x_1)^2$, we deduce that $\theta(c_1)=1$.
\end{enumerate}
Hence $\theta(x_n)=\theta(c_n)=1$ for all $n$. Finally, from relation (3.1) we deduce that $y_ny_{n+1}^{-1}$ also maps to 1, and from here that $\theta(y_iy_j^{-1})=1$ for all $i,j$. From relation (3.2) with $n=1$ we deduce that $z_0=y_0y_2$ is also in the kernel, and by induction for any other $z_n$ from $z_{n-1}=y_{2n-2}y_{2n}$ and $y_{2n-2}^{-1}y_{2n+2}$. So $\theta$ is the trivial map and $\txz=N$.\qed

\section{A presentation for \vt}

The construction of a generating set for \vt\ is straightforward. Following \cite{cfp} and in the same way that we have done for \ttt, we need to add a series of generators $\pi_n$, for $n\ge 0$, which introduce noncyclic permutations to the leaves of a spine. The generator $\pi_n$ has $n+2$ $x$-carets, hence it has $n+3$ leaves, and the permutation just transposes leaves $n+1$ and $n+2$. This is analogous to the generators for $V$ in \cite{cfp}. See the generators $\pi_n$ in Figure \ref{pigens}. It is clear that the series $x_n$, $y_n$, $c_n$ and $\pi_n$ generate \vt. Observe that we write the element in $p\pi q^{-1}$ form and then the permutation $\pi$ is in some $\mathcal S_k$, which is generated by a $k$-cycle and a transposition. Suppose $k\ge 3$, the case $k=2$ is straightforward and is actually in \ttt. Then the $k$-cycle is $c_{k-2}$ and the transposition is $\pi_{k-3}$.

\begin{figure}
  \centerline{\includegraphics[width=60mm]{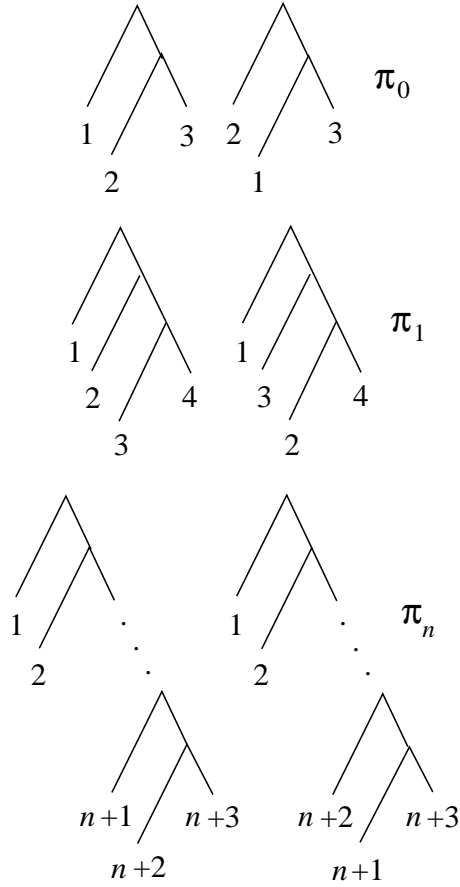}}
  \caption{The $\pi$ generators in \ttt.}\label{pigens}
\end{figure}

Hence, we need to find relations involving these generators. As before, the relations will be obtained by seeing how the generators of different types interact with each other. First of all, see Theorem \ref{prez}, involving $x_n$, $y_n$ and $c_n$, we have all relators from \ttt:

\begin{enumerate}
\item[(1)] The relations involving $x_n$ and $y_n$:
\begin{enumerate}
\item[(1.1)] $x_jx_i=x_ix_{j+1}$, if $i<j$.
\item[(1.2)] $x_jy_i=y_ix_{j+1}$, if $i<j$.
\item[(1.3)] $y_jx_i=x_iy_{j+1}$, if $i<j$.
\item[(1.4)] $y_jy_i=y_iy_{j+1}$, if $i<j$.
\item[(1.5)] $y_n^2=x_nx_{n+1}$.
\end{enumerate}
\item[(2)] The relations involving $x_n$ and $c_n$:
\begin{enumerate}
\item[(2.1)] $x_kc_{n+1}=c_{n}x_{k+1}$, if $k<n$.
\item[(2.2)] $c_nx_0=c_{n+1}^2$.
\item[(2.3)] $c_n=x_nc_{n+1}$.
\item[(2.4)] $c_n^{n+2}=1$.
\end{enumerate}
\item[(3)] The relations involving $c_n$ and $y_n$:
\begin{enumerate}
\item[(3.1)] $y_kc_{n+1}=c_{n}y_{k+1}$, if $k<n$.
\item[(3.2)] $c_ny_0=y_{n+1}^{-1}c_{n+1}^2$.
\end{enumerate}
\end{enumerate}

Observe that the combinatorics involving $x_n$, $c_n$ and $\pi_n$ are the exact same as in $V$, the only difference is that a binary caret is replaced by an $x$-caret, but the methods are exactly the same. Hence, we have all relators from $V$
(as before when we had all relators for $T$ in \ttt).

\begin{enumerate}
\item[(4)] The relations involving $x_n$ and $\pi_n$:
\begin{enumerate}
\item[(4.1)] $\pi_i x_j = x_j\pi_i$ if $j \geq i+2.$
\item[(4.2)] $\pi_i x_{i+1} = x_i\pi_{i+1}\pi_i.$
\item[(4.3)] $\pi_i x_i = x_{i+1}\pi_i\pi_{i+1}.$
\item[(4.4)] $\pi_i x_j = x_j \pi_{i+1}$ if $0 \leq j <i.$
\item[(4.5)] $\pi_i^2 =1.$
\item[(4.6)] $(\pi_{i+1}\pi_i)^3 =1.$
\item[(4.7)] $\pi_i\pi_j=\pi_j\pi_i$ if $j \geq i+2.$
\end{enumerate}
\item[(5)] The relations involving $c_n$ and $\pi_n$:
\begin{enumerate}
\item[(5.1)] $c_n\pi_k = \pi_{k-1}c_n.$
\item[(5.2)] $c_n\pi_0 = \pi_0 \cdots \pi_{n-1}c_n^2.$
\item[(5.3)] $c_n^2\pi_0 = \pi_{n-1}\cdots \pi_0 c_n.$
\item[(5.4)] $c_n^3\pi_0 = \pi_{n-1}c_n^3.$
\end{enumerate}
\end{enumerate}

So as was the case in \ttt, the only relators we need to add are those which involve $y_n$ and $\pi_n$. But the case here differs from what we have found in \ttt, because the transposition in $\pi_n$ never involves the last leaf of the caret. in \ttt, when we combined $y_n$ with $c_n$, there are some cases where the $y$-caret ends up being on the spine, the last caret of the tree, and since we are taking the spine with only $x$-carets, another $y$-caret had to be added to switch. This is the reason why relations (3.1) and (3.2) have an extra $y$-generator than their counterparts with $x_n$, which are (2.1) and (2.2). But this does not happen here, the relations involving $y_n$ and $\pi_n$ look exactly the same as for $x_n$, so they are the same as (4.1)-(4.4) but replacing $x$ by $y$:

\begin{enumerate}
\item[(6)] The relations involving $y_n$ and $\pi_n$:
\begin{enumerate}
\item[(6.1)] $\pi_i y_j = y_j\pi_i$ if $j \geq i+2.$
\item[(6.2)] $\pi_i y_{i+1} = y_i\pi_{i+1}\pi_i.$
\item[(6.3)] $\pi_i y_i = y_{i+1}\pi_i\pi_{i+1}.$
\item[(6.4)] $\pi_i y_j = y_j \pi_{i+1}$ if $0 \leq j <i.$
\end{enumerate}
\end{enumerate}

And this fact is the reason why now the proof of the fact that these are all the relations we need is straightforward. For the interactions between $x_n$, $y_n$ and $c_n$ we already have the \ttt-relators, and the way the $\pi_n$ interact with the other three series of generators is exactly the same as in $V$. Hence, all results for $V$ in \cite{cfp} apply here, and clearly an element can be put in $p\pi q^{-1}$ form using all these relations. Once it is in this form, if it was the identity, it is in \ft\ and it is a product of the relators (1.1)-(1.5). This finishes the construction of the presentation for \vt.

\section{The group $V_{xz}$ is simple}\label{index2-sec}
Inspecting the presentation for \vt\ we observe the same phenomenon as in \ft\ and \ttt, namely, that the $y$-generators only appear in the relators an even number of times. Hence there is a well-defined epimorphism
$$ \varphi: \vt \longrightarrow \quot{\zz}{2\zz},$$
given by $\varphi(x_n)=\varphi(c_n)=\varphi(\pi_n)=0,$ and $\varphi(y_n)=1.$ The kernel of this epimorphism is generated by those elements in $\vt$ where $y$-generators appear an even number of times.

In particular, $\vt$ is not simple. We will consider the kernel of the homomorphism $\varphi$:

\begin{defn} We call $\vxz$ the kernel of the map
$$ \varphi: \vt \longrightarrow \quot{\zz}{2\zz}.$$
\end{defn}

\begin{thm}\label{thethm}
The group $\vxz$ is a simple group.
\end{thm}

Here we follow an argument described by Brin in \cite{brin2v}, and attributed to M. Rubin.

\begin{lem}\label{perm-lem}
$\vxz$ is generated by permutations.
\end{lem}

\begin{proof} Let $v \in \vxz.$ Then, by Lemma \ref{el-vt-lem} and the parity of $y$, it can be represented by a triple $(T,\pi ,S)$ such that $T$ and $S$ have $k$ leaves, $S$ has no $y$-carets and the only $y$-carets in $T$ are an even number  with no left children, and $\pi \in \mathcal S_k.$

\begin{figure}
  \centerline{\includegraphics[width=85mm]{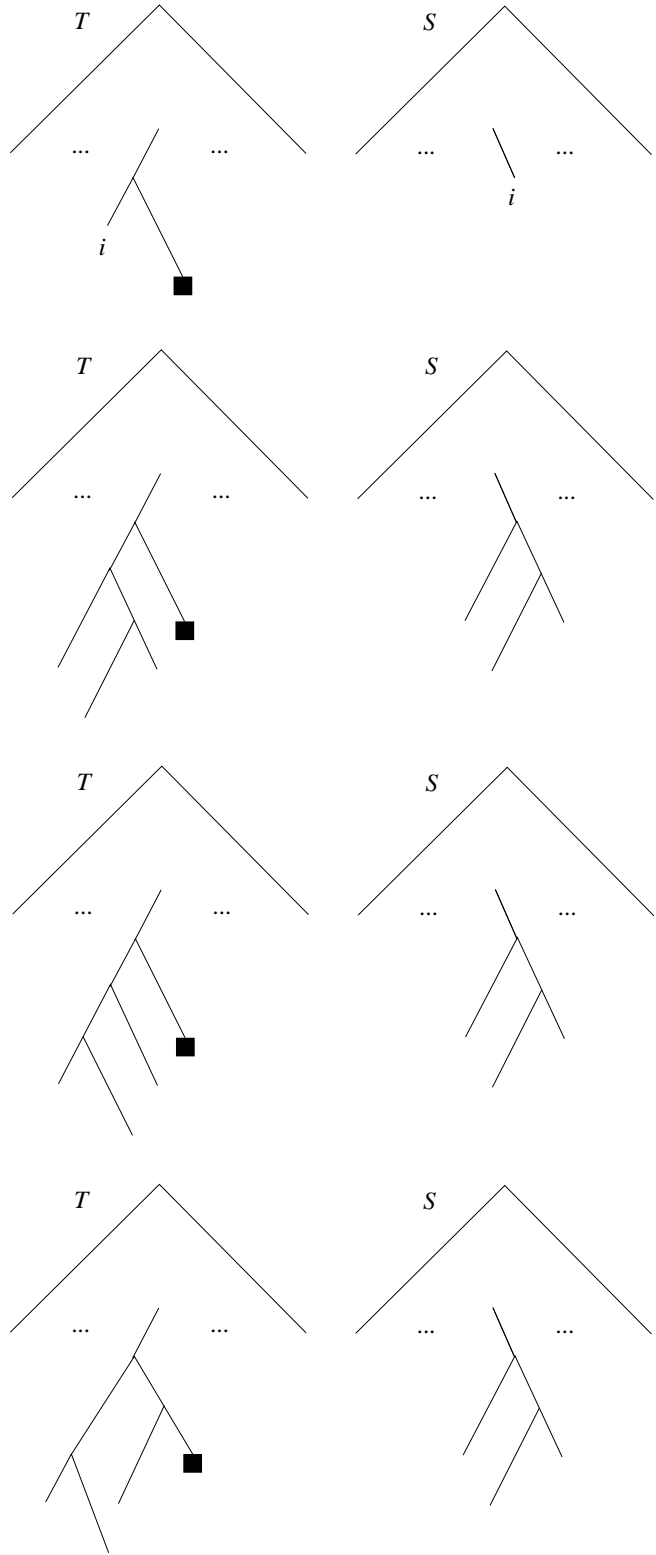}}
  \caption{The process of creating an exposed $y$-caret in the proof of Lemma \ref{perm-lem}. The black box indicates that a subtree is present in that leaf.}\label{exposed}
\end{figure}

To begin the process we need an exposed $y$-caret in the tree $T$, i.e. a $y$-caret with no children. Observe that we only know that the $y$-carets have no left children, but they could have right ones. Take any $y$-caret in $T$ and make it exposed the following way. See Figure \ref{exposed} for reference. Add two $x$-carets to its left leaf (and the corresponding ones in $S$). Perform a basic move on these two $x$-carets to transform them to $y$-carets, and this produces three consecutive $y$-carets. Now switch the two top ones back to $x$-type, again using a basic move. This makes the bottom $y$-caret exposed. The number of $y$-carets has not changed. Note that this process is the reverse of a \emph{hidden cancellation} described in \cite[Section 6]{ftau}.

Now we have ensured that in $T$ there is at least one exposed caret.
Pick a pair of $y$-carets $c_1$ and $c_2$ in $T$ and assume $c_1$ is exposed.
Then add a $y$-caret to the left leaf of $c_2$; this means that we have to add a $y$-caret $d_1$ to the corresponding leaf in $S$. Switch $c_2$ and the recently added caret to type $x$ using a basic move.
We have now a representation of $v$ by a triple $(T', \pi', S'),$ where $T'$ and $S'$ are now trees with $k+1$ carets, and $\pi' \in \mathcal S_{k+1},$ and the carets $c_1$ and $d_1$ are both exposed, one in each tree.

We now precompose with a permutation $\sigma$ of the leaves of $T',$ so that $c_1$ in the left-most tree has the same labelling as $d_1$ in $S'.$ We now have obtained an element $v_1=\sigma v$ represented by
$$(T',\sigma, T')(T', \pi' , S'),$$
where the $y$-carets $c_1$ and $d_1$ map to each other, so they have become redundant and can be removed. Hence $v_1$ is now represented by a triple as above with $k+2$ leaves as above, but with two fewer $y$-carets.

We repeat this process until we have $u = \rho v,$ where $\rho$ is a product of permutations and $u$ is an element represented by a tree-pair with only $x$-carets.

We can now follow the procedure set out by Brin as if we were in $V$: find a pair of exposed carets, one in each tree, premultiply with a permutation making this pair of carets redundant, remove the carets to obtain an element represented by a tree with fewer carets. Then continue this process until no tree is left. Our original element is now a product of permutations.
\end{proof}

\begin{defn}
We say an element $\pi \in \vt$ is a proper transposition if $\pi$ is a transposition which can be represented by a triple
$(T, \pi, T),$ where $T$ is a tree with at least three carets.
\end{defn}

\begin{cor}
$\vxz$ is generated by proper transpositions.
\end{cor}

\begin{proof} Suppose we have a transposition represented by a tree-pair with two or fewer carets, then adding redundant carets simultaneously on both sides gives a permutation with the right number of carets, which, in turn is a product of proper transpositions.
\end{proof}

\begin{lem}
Any two proper transpositions are conjugate in $\vxz.$
\end{lem}

\begin{proof} Let $\pi_1$ and $\pi_2$ be two proper transpositions represented by $(T,\pi_1,T)$ and $(S, \pi_2, S)$ respectively.
Since both transpositions are proper, we have ``uninvolved" leaves in both trees, so we can add redundant carets to ensure that $T$ and $S$ have the same number of leaves. Now the element $v = (T, \sigma, S),$ where $\sigma$ maps the leaves involved in $\pi_1$ to the leaves involved in $\pi_2$ conjugates $\pi_2$ into $\pi_1.$ A priori this is an element of $\vt.$  Suppose that both $S$ and $T$ have the same parity of $y$-leaves. Then, automatically, $v \in \vxz.$ Should the parity of $y$-leaves in $T$ and $S$ differ, we can always add a redundant $y$-caret to $T$ and a corresponding redundant $x$-caret to $S$ to obtain a tree-pair $(T', \sigma' ,S')$ representing an element of $\vxz$ conjugating $\pi_2$ into $\pi_1.$
\end{proof}

\begin{lem}\label{normal-lem}
Any normal subgroup $N$ of $\vxz$ contains a proper transposition.
\end{lem}

\begin{proof} This proof is exactly the proof of Brin \cite{brin2v}, so we only provide an outline here. Let $v  = (T, \pi, S) \in N.$ By expanding we can always represent $v$ by a tree-pair with at least three carets, and so that an interval gets completely moved off itself. We now consider all elements as moving nodes in an infinite tree.  In particular, there is a node $l$ mapping to a node $k$ so that neither is in a subtree containing the other. We expand $l$ and $k$ with a caret of the same type, an $x$-caret, say. Let $g$ be the element switching $(l_o,l_1),$ the leaves of the caret starting in $l$. Then $u = [g,v] \in N$ and is a product of two transpositions. Now, set $h$ to be the transposition switching $l_0$ with $k_0$, the left-hand leaf of the caret starting in $k.$ Then $w = [h,u] \in N$ is a permutation swapping $l_0$ with $k_0$ and $l_1$ with $k_1$, hence $w$ is a transposition switching $l$ with $k$.

\end{proof}

This sequence of results, \ref{perm-lem} through to \ref{normal-lem}, now yield Theorem \ref{thethm}.

\section{A $V$-type Thompson's group with a normal subgroup of index $4.$}

As mentioned above, Higman \cite{higman} showed that some of the $V$-type groups have index-$2$ normal subgroups. This, however, occurs only when the  $n$-ary carets in the trees representing the elements have $n$ being odd. Hence, this is a very different phenomenon than then one we observe for $\vt.$  Let $\beta = \sqrt2-1,$ i.e. the positive root of $X^2+2X-1.$  In \cite{clearyirrold}, Cleary also showed that the group $G(I, \zz[\beta], \langle \beta \rangle) = \fb$ can be seen as a group of rearrangements of $\beta$-regular subdivisions of the unit interval. In his thesis, J. Brown \cite{jb}, shows that elements of $\fb$ can also be represented by tree-pair diagrams, this time having three caret-types with three legs each, see Figure \ref{carets}.

\begin{figure}
\centerline{\includegraphics[width=80mm]{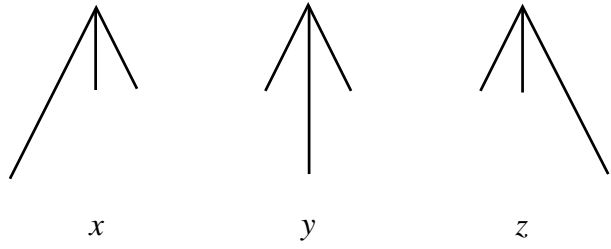}}
\caption{The three different caret types in $V_\beta$.}
\label{carets}
\end{figure}

\noindent He also produced an infinite presentation involving $x$-carets and $y$-carets only:

\begin{prop}\cite[6.1.1.]{jb} A presentation for $\fb$ is:
$$\left< x_n,y_n,\mbox{ for }n\ge 0\,\middle|\,
\begin{array}
l a_ib_j = b_ja_{i+2},\,\mbox{ for all } i>j \mbox{, }\mbox{and } a,b \in \{x,y\},\\
y_k^2 = x_kx_{k+1},\,\mbox{ for }k \geq 0
\end{array}
\right>.
$$
\end{prop}

The generators are represented by the following tree-pairs $(T_{a_i}, S_k),$ where $T_{a_i}$ are as in Figure \ref{gens} ($a \in \{x,y\}$), and $S_k$ is a spine with $\lfloor\frac{i}{2}\rfloor+2$ $x$-carets.

\begin{figure}
\centerline{\includegraphics[width=100mm]{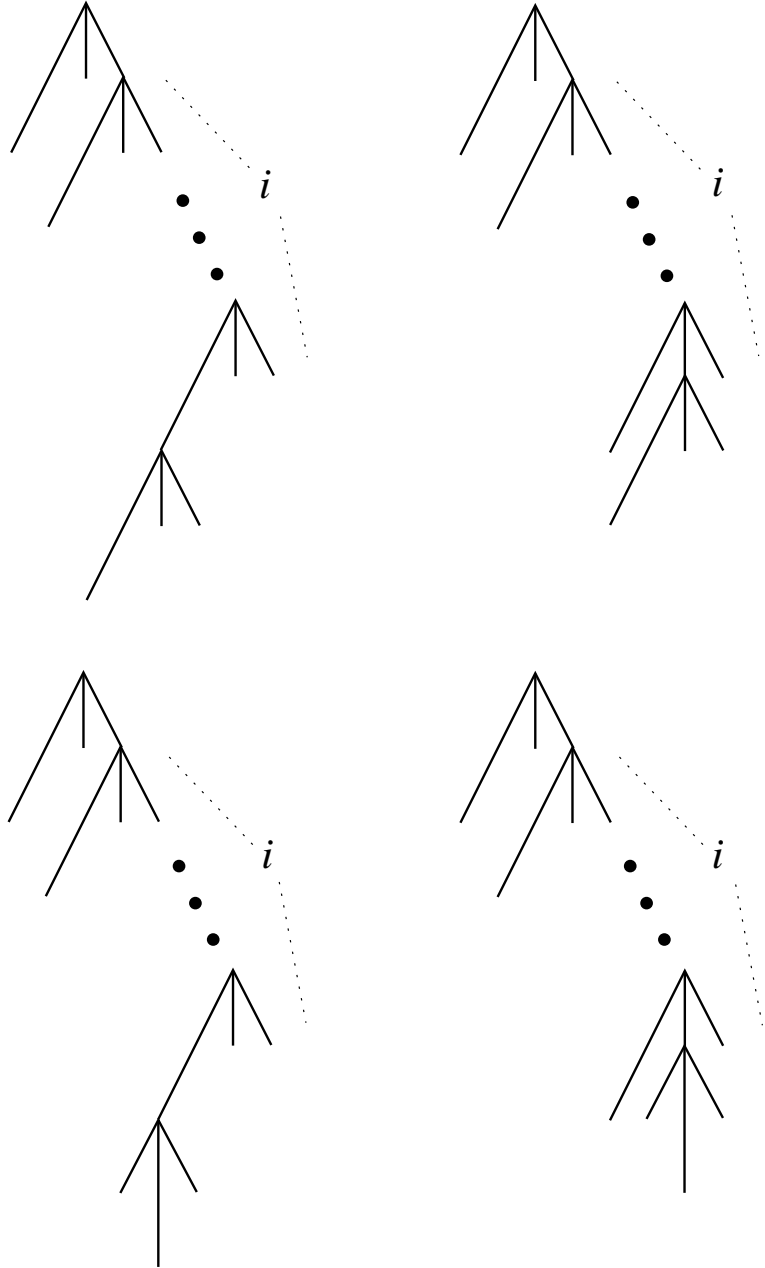}}
\caption{The left-hand trees for the generators. Top left $T_{x_{2i}}$, top right $T_{x_{2i+1}}$, bottom left $T_{y_{2i}}$, bottom right $T_{y_{2i+1}}$}
\label{gens}
\end{figure}

\noindent Analogously to $\vt$ we can now define the group $\vb,$ whose elements are represented by triples $(T, \pi, S)$, where
$T$ and $S$ are trees with the same number of leaves, and $\pi$ is a permutation of leaves.  We now follow the arguments of Section \ref{index2-sec}.

\begin{lem}\label{vb-properties} We have the following:
\begin{enumerate}
\item Every element $v \in \vb$ has an expression as $p\sigma q^{-1}$, where $p$ and $q$ are words in positive powers of $x_i$ and $y_i$ ($i \geq 0$), and $\sigma$ is a permutation of the leaves of a spine with only $x$-carets.
\item In every such expression of $v \in \vb$ as a word in $x_i,y_j, \pi_k,$ ($i,j,k\geq 0$) the parity of the sum of the exponents of the $y_j$ is constant.
\end{enumerate}
\end{lem}

\begin{proof} The first statement is immediate by the usual method of introducing a spine in the middle and permuting the leaves there. The second statement is a check similar to that for $\vt$.
\end{proof}

\begin{lem} There is a well defined surjective homomorphism
$$\varphi: \vb \longrightarrow \quot{\zz}{2\zz}$$
given by $\varphi(x_i) = \varphi(\sigma) =0$ and $\varphi(y_i)=1$ ($i \geq 0$), where $\sigma$ is a permutation of leaves.
\qed
\end{lem}

Analogously to Higman's original argument \cite[Section 5]{higman}, we can also find a subgroup of index-$2$ in $\vb$ given by the sign of the permutation in $(T, \pi, S).$ This sign is unchanged by expansion of the trees, see \cite[Section 5]{higman}, and also by the relations between the different carets, which follows from the fact that the relations do not change the order of the leaves, see Figure \ref{rel}.

\begin{figure}
\centerline{\includegraphics{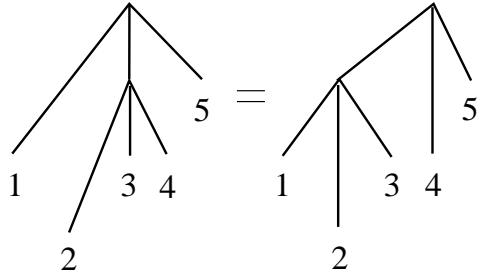}}
\caption{The order of leaves doesn't change when we perform a basic move replacing two carets by another two giving the same subdivision.}
\label{rel}
\end{figure}

\begin{lem}
There is a well-defined subgroup $\vb^+$ of index $2$ in $\vb$ given by the elements of even parity. Hence there is a canonical projection:
$$\rho: \vb \longrightarrow \quot{\vb}{\vb^+} \cong \quot{\zz}{2\zz}.$$
\qed
\end{lem}

\begin{thm}
$\vb$ has an index-$4$ normal subgroup.
\end{thm}

\begin{proof}
We have a surjection $\vb \rightarrow \quot{\zz}{2\zz} \times \quot{\zz}{2\zz}$ given by
$v \mapsto (\rho(v),\varphi(v)).$
\end{proof}

The kernel $K$ of this map is the subgroup given by all even elements, which have an even number of $y_i$ in each expression $v = p\sigma q^{-1}$ where $p$ and $q$ are words in positive powers of $x_i$ and $y_i$ ($i \geq 0$). We suspect that $K$ is simple, but the methods above for $\vt$ do not work, and one would have to follow Higman's original proof of the simplicity of $G_{n,r}^+$.
This is the subject of N. Winstone's thesis \cite{nick}, where Thompson groups for more general algebraic numbers are considered.

\bibliography{pepsrefs}
\bibliographystyle{plain}

\end{document}